\documentclass[12pt,twoside,final]{article}
\usepackage{chngpage}
\usepackage[a4paper,left=22mm,right=22mm,top=30mm,bottom=30mm]{geometry}
\usepackage{imakeidx}
\usepackage{xr-hyper}
\usepackage{xr}
\usepackage[all]{xy}
\usepackage[center]{titlesec}
\titleformat*{\section}{\large\bfseries}
\usepackage{enumerate,cite}
\usepackage[latin1]{inputenc}
\usepackage{amssymb}
\usepackage{amsthm}
\usepackage{amsmath}
\usepackage{amsfonts}
\usepackage{mathrsfs}
\usepackage{hyperref}
\usepackage{color}
\usepackage{graphicx}
\usepackage{setspace}
\usepackage{bm}
\usepackage{enumitem}
\usepackage{booktabs}
\usepackage[capitalize]{cleveref}
\usepackage{tikz}
\usetikzlibrary{matrix,chains}
\usepackage[center]{titlesec}
\titleformat*{\section}{\normalsize\bfseries}
\usepackage{indentfirst}

\usepackage{amscd}

\newcommand{\Aut}{\operatorname{Aut}\nolimits}

\newcommand{\IBr}{\operatorname{IBr}\nolimits}

\theoremstyle{remark}

\theoremstyle{definition}

\theoremstyle{plain}
\newtheorem{thm}{Theorem}[section]

\newtheorem{lem}[thm]{Lemma}

\newtheorem{cor}[thm]{Corollary}

\newtheorem{prop}[thm]{Proposition}

\newtheorem{order}[thm]{}

\newtheorem{conj*}{Conjecture}
\numberwithin{equation}{thm}

\usepackage{lastpage}
\usepackage{fancyhdr}

\begin{document}

\begin{center}{\Large\bf 2-blocks with abelian defect groups and inertial quotient of prime order }
\bigskip{\large}

\bigskip{Qianhu Zhou and Kun Zhang}

\bigskip{\scriptsize 1.School of Mathematics and Statistics, Central China Normal University, Wuhan, 430079, China
\\2.Faculty of Mathematics and Statistics, Hubei University, Wuhan, 430062, China}
\end{center}

{\noindent\small{\bf Abstract}
In this paper, we classify all $2$-blocks for which the defect groups are abelian and the inertial quotient has prime order.
As a consequence, we prove that Brou\'e's abelian defect group conjecture holds for all blocks under consideration here. }


\section{Introduction}
Let $p$ be a prime and $(\mathcal{K}, \mathcal{O}, k)$ a complete $p$-modular system. We assume that the residue field $k$ is
algebraically closed of characteristic $p$ and the field of fractions $K$ is of characteristic 0. Throughout this
paper, we assume that $\mathcal{K}$ is large enough for the finite groups considered below.

Let $G$ be a finite group. A $p$-block $b$ of $G$ is a primitive central idempotent of
the group algebra ${\cal O}G$. Let $P$ be a defect group of $b$ and  $b_0$ the Brauer correspondent of $b$ in the normalizer ${\rm N}_G(P)$.
A key goal in block theory is to understand the connection between the block algebras
 ${\cal O}Gb$ and ${\cal O}{\rm N}_G(P)b_0$.
 The celebrated \textit{Brou\'e's abelian defect group conjecture} says that the block algebras ${\cal O}Gb$ and ${\cal O}{\rm N}_G(P)b_0$ are derived equivalent when the defect group $P$ is abelian(see \cite[6.1]{B}).

Following \cite{P2011}, a block is called \textit{inertial} if $\mathcal{O}Gb$ and $\mathcal{O}{\rm N}_G(P)b_0$ are basically Morita equivalent \cite{P1999}, and it is well-known that Brou\'e's conjecture holds for such blocks.
Numerous examples of inertial blocks with abelian defect groups are known, such as blocks of $p$-solvable groups, 2-blocks with defect groups isomorphic to $C_{2^{n_1}} \times C_{2^{n_2}} \times \cdots \times C_{2^{n_t}}$ where $n_i \geq 2$ for all $i$ \cite{ZZ1}, blocks with only one orbit of irreducible Brauer characters \cite{Z}, and so on.

The classification of $2$-blocks with abelian defect groups for finite quasi-simple groups has been established (see \cite[Theorem 6.1]{EKK} or Theorem \ref{2-block-abel} below).
This classification has been used to determine the Morita equivalence classes of $2$-blocks with  certain abelian defect groups \cite{CGA, EL2, EKK, ECW2016, EM, WZZ}. This, in turn, allows for a proof of Brou\'e's conjecture in these cases.

Let $(P,e_P)$ be a maximal $b$-Brauer pair \cite[Section 6.3]{L}.  We denote by ${\rm N}_G(P,e_P)$ the stabilizer of $(P,e_P)$ in $G$ under conjugation. The $\textit{inertial quotient}$ of $b$ is then defined as the quotient group $\mathbb{E}(b)={\rm N}_G(P, e_P)/P{\rm C}_G(P)$. In particular, when $P$ is abelian, $\mathbb{E}(b)$ is the quotient group ${\rm N}_G(P, e_P)/{\rm C}_G(P)$. Then it is a group of automorphisms of $P$,
and the \textit{hyperfocal subgroup} (see\cite[1.7]{P2}) of $b$ with respect to $(P,e_P)$ is equal to $[P,\mathbb{E}(b)]$.

We focus on $2$-blocks with abelian defect groups whose inertial quotients have prime order, and prove the following classification.

\begin{thm}\label{Main1}
Let $b$ be a $2$-block of a finite group $G$ with an abelian defect group $P$ and with $\mathbb{E}(b)$ of prime order. Then one (or more) of the following holds:
\begin{enumerate}[label=(\roman*), itemsep=-3pt, topsep=7pt]
    \item $b$ is inertial;
    \item the hyperfocal subgroup of $b$ is a Klein four-group;
    \item $b$ is basic Morita equivalent to the principal block of $A_1(2^a) \times R$, where $2^a - 1$ is prime and $R$ is an abelian $2$-group.
\end{enumerate}
\end{thm}

As a consequence, we obtain the following corollary.

\begin{cor}\label{Main2}
Brou\'e's abelian defect group conjecture is true for the blocks considered in Theorem \ref{Main1}.
\end{cor}

\section{Reduced blocks}
In this section, we consider reduced blocks with abelian defect groups.  In general, we follow the notation and
terminology of \cite{N1}

We recall that a block $b$ of $G$ is called \emph{quasi-primitive} if every block of a normal subgroup of $G$ covered by $b$ is $G$-invariant. Further, $b$ is said to be \emph{reduced} (see \cite[Proposition 6.1]{A}) if it is quasi-primitive and satisfies the following additional condition: for any normal subgroup $N \trianglelefteq G$ such that $b$ covers a nilpotent block of $N$, we have
\[
N \leq {\rm Z}(G){\rm O}_p(G) \quad \text{and} \quad {\rm O}_{p'}(N) \leq {\rm Z}(G) \cap [G,G],
\]
where ${\rm Z}(G)$ is the center of $G$, ${\rm O}_p(G)$ is the maximal normal $p$-subgroup, ${\rm O}_{p'}(N)$ is the maximal normal $p'$-subgroup of $N$, and $[G,G]$ is the commutator of $G$.

As usual, ${\rm E}(G)$ denotes the layer of $G$, ${\rm F}(G)$ the Fitting subgroup, and ${\rm F}^*(G) = {\rm F}(G){\rm E}(G)$ the generalized Fitting subgroup. Reduced blocks with abelian defect groups were characterized in \cite{ZZ1}, leading to the following structural classification:

\begin{thm}[{\cite[Theorem 2.1]{ZZ1}}]\label{Reduced-block}
Let $b$ be a block of $G$ with defect group $P$. If $b$ is reduced and $P$ is abelian, then one of the following holds:
\begin{enumerate}[itemsep=-3pt,topsep=7pt,label=\emph{(\roman*)}]
\item $P$ is normal in $G$.
\item The layer ${\rm E}(G)$ is non-trivial, and there exists normal subgroups $M$ and $L$ of $G$
such that ${\rm F}^*(G)\leq M\leq L$, such that $G/L$ and $M/{\rm F}^*(G)$ are both $p'$-groups, such that $L=PM$, and such that each component of $G$ is normal in $L$.
\end{enumerate}
\end{thm}

 A block $b$ of $G$  is \textit{nilpotent covered} if there is a finite group $\tilde{G}$ with the normal subgroup $G$ and a nilpotent block $\tilde{b}$ of $\tilde{G}$ covering $b$ (i.e. $\tilde{b}\cdot b\neq0$).
 A subnormal quasi-simple subgroup of a finite group $G$ is called a $\textit{component}$ of $G$. The layer ${\rm E}(G)$ is the central product of all components of $G$.

 Let $P$ be a defect group of $b$.
 In the remainder of this section, we always assume that $b$ is reduced, and that $P$ is abelian which is not normal in $G$. For any component $X$ of $G$, let $b_X$ be a block of $X$ such that $b_X\cdot b\neq0$. Let $b_{{\rm E}(G)}$ be the  block of ${\rm E}(G)$ covered by $b$. It is obvious that $b_X$ is the unique block of $X$ covered by $b_{{\rm E}(G)}$.

\begin{prop}\label{Inertial}
Keep the notation and assumptions above.  If the block $b_X$ is nilpotent covered for every component $X$ of $G$,
then $b$ is inertial.
\end{prop}

\begin{proof}
By Theorem \ref{Reduced-block}(ii), $X$ is invariant under the conjugation action of $P$. So the block $b_X$ can be viewed as a block of $PX$,
and $P$ serves as a defect group of it. Since $b_X$ is nilpotent covered, the set $\IBr(b_X)$, which denotes the set of all irreducible Brauer characters belonging  to $b_X$, forms a single orbit under the action of $\Aut(X)_{b_X}$, where $\Aut(X)_{b_X}$ is the stabilizer of $b_X$ in the automorphism group $\Aut(X)$ under conjugation. Thus by \cite[Theorem 1.1]{Z},  $b_X$ is inertial as a block of $PX$.
Finally, by \cite[Corollary 2.4]{ZZ1} and the arbitrariness of $X$, it follows that $b$ is inertial.
\end{proof}

\begin{lem}\label{G=N}
Keep the notation and assumptions above. Let $N\unlhd G$ such that ${\rm C}_G(P)\leq N$. If the inertial quotient $\mathbb{E}(b)$ has prime order, then $G=N$.
\end{lem}

\begin{proof}
Let $c$ be a block of $N$ covered by the block $b$. Since $P$ is abelian and ${\rm C}_G(P)\leq N$, we have $P\leq N$.
Since $b$ is reduced, $c$ is $G$-invariant. Thus \cite[Theorem 9.26]{N1} tells us that $P$ is also a defect group of  $c$.

Fix a common maximal Brauer pair $(P,e_P)$ for $b$ and $c$. Note that the inertial quotient $\mathbb{E}(c)={\rm N}_N(P,e_P)/{\rm C}_N(P)=({\rm N}_G(P,e_P)\cap N)/{\rm C}_G(P)$. So $\mathbb{E}(c)$ is a subgroup of  $\mathbb{E}(b)$. Since $\mathbb{E}(b)$ has prime order, $\mathbb{E}(c)$ is either trivial or equal to $\mathbb{E}(b)$. Suppose that $\mathbb{E}(c)$ is  trivial. Then $c$ is nilpotent  by \cite[Proposition 8.11.3]{L}.
Since $b$ is reduced, we have $N\leq {\rm Z}(G){\rm O}_p(G)$. Thus $P={\rm O}_p(G)$, which contradicts
our assumption that $P$ is not normal in $G$. So $\mathbb{E}(c)$ is equal to $\mathbb{E}(b)$, and then ${\rm N}_G(P,e_P)\leq N$. Also, the Frattini argument tells us that $G=N{\rm N}_G(P,e_P)$, so $G=N$.
\end{proof}

\begin{cor}\label{ZG}
Keep the notation and assumptions above. If the inertial quotient $\mathbb{E}(b)$ has prime order, then ${\rm Z}(G) = {\rm O}_p(G){\rm O}_{p'}(G)$.
\end{cor}

\begin{proof}
As $b$ is reduced, ${\rm O}_{p'}(G) \leq {\rm Z}(G)$.
So we may assume that ${\rm O}_p(G) \neq 1$. Let $N={\rm C}_G({\rm O}_p(G))$. Since any defect group of $b$ contains ${\rm O}_p(G)$, then $N$ is normal in $G$ and contains $C_G(P)$.  By Lemma \ref{G=N}, we have $G = N$, so ${\rm O}_p(G) \leq {\rm Z}(G)$.
It follows that
\(
\mathrm{Z}(G) = \mathrm{O}_p(G) \, \mathrm{O}_{p'}(G).
\)
\end{proof}

\begin{cor}\label{X_i is normal}
Keep the notation and assumptions above. If the inertial quotient $\mathbb{E}(b)$ has prime order, then every component of  $G$ is normal in $G$.
\end{cor}
\begin{proof}
Since $\mathrm{E}(G)$ is normal in $G$ and the block $b$ is reduced, the block of $\mathrm{E}(G)$ covered by $b$ is unique; denote it by $b_{\mathrm{E}(G)}$. Set $P_{\mathrm{E}(G)} = P \cap \mathrm{E}(G)$. By \cite[Theorem 9.26]{N1}, $P_{\mathrm{E}(G)}$ is a defect group of $b_{\mathrm{E}(G)}$. Let $(P_{\mathrm{E}(G)}, f_{P_{\mathrm{E}(G)}})$ be a maximal $b_{\mathrm{E}(G)}$-Brauer pair. The Frattini argument then implies
\[
G = \mathrm{E}(G) \, \mathrm{N}_G(P_{\mathrm{E}(G)}, f_{P_{\mathrm{E}(G)}}).
\]
Let $K = \mathrm{E}(G) \, \mathrm{C}_G(P_{\mathrm{E}(G)})$. Then $K$ is normal in $G$. Clearly, $\mathrm{C}_G(P) \leq \mathrm{C}_G(P_{\mathrm{E}(G)}) \leq K$. Applying Lemma \ref{G=N} to $K$ yields $G = K$.

Let $X$ be a component of $G$ and let $b_X$ be the block of $X$ covered by $b_{\mathrm{E}(G)}$. By \cite[Theorem 9.26]{N1}, $P_X = P_{\mathrm{E}(G)} \cap X$ is a defect group of $b_X$. Take an arbitrary element $g \in \mathrm{C}_G(P_{\mathrm{E}(G)})$. Then $P_X \leq X \cap X^g$.

Suppose $X \neq X^g$. Then $P_X \leq \mathrm{Z}(X)$, so $b_X$ has a central defect group. Hence, the block of every $G$-conjugate of $X$ covered by $b_{\mathrm{E}(G)}$ also has a central defect group. Let $H$ be the minimal normal subgroup of $G$ containing $X$, which is a central product of all $G$-conjugates of $X$. Then, by \cite[Corollary 2.9]{MNST}, the block of $H$ covered by $b$ has a central defect group.
By Corollary \ref{X_i is normal}, ${\rm O}_p(G){\rm O}_{p'}(G)=\mathrm{Z}(G).$
Since $b$ is reduced, it follows that $X\leq H \leq{\rm O}_p(G){\rm O}_{p'}(G)=\mathrm{Z}(G)$, which is a contradiction. Therefore, $X^g = X$ for all $g \in \mathrm{C}_G(P_{\mathrm{E}(G)})$, and since $G = \mathrm{E}(G) \mathrm{C}_G(P_{\mathrm{E}(G)})$,  $X$ is normal in $G$.
\end{proof}

\begin{lem}\label{NGPeP}
Keep the notation and assumptions above. Let $X$ be a component of $G$ and $b_X$ the block of $X$ covered by $b$. Let $(P,e_P)$ be a maximal $b$-Brauer pair. If $b_X$ is non-nilpotent covered and the inertial quotient $\mathbb{E}(b)$ has prime order, then
$
\mathrm{N}_G(P,e_P) = \mathrm{N}_X(P,e_P) \mathrm{C}_G(P).
$
\end{lem}

\begin{proof}
Let $H = X \mathrm{C}_G(P)$. By Corollary \ref{X_i is normal}, $X$ is normal in $H$, and clearly $\mathrm{C}_H(P) = \mathrm{C}_G(P)$. Let $b_H$ be the induced block $(e_P)^H$. Then $b_H$ covers $b_X$, $P$ is a defect group of $b_H$, and so $(P,e_P)$ is also a maximal $b_H$-Brauer pair. Since $\mathbb{E}(b)$ has prime order, the intersection $\mathrm{N}_G(P,e_P) \cap H$ is either $\mathrm{C}_G(P)$ or $\mathrm{N}_G(P,e_P)$.

If $\mathrm{N}_G(P,e_P) \cap H = \mathrm{C}_G(P)$, then the inertial quotient $\mathbb{E}(b_H)$ is trivial, so $b_H$ is nilpotent by \cite[Proposition 8.11.3]{L}, contradicting the assumption that $b_X$ is non-nilpotent covered. Hence, $\mathrm{N}_G(P,e_P) \cap H = \mathrm{N}_G(P,e_P)$, and the conclusion follows.
\end{proof}

\begin{lem}\label{Iso}
Keep the notation and assumptions above. Let $X$ be a component of $G$ and $b_X$ the block of $X$ covered by $b$. Assume that the inertial quotients $\mathbb{E}(b)$ and $\mathbb{E}(b_X)$ have prime order, that  the outer automorphism group $\mathrm{Out}(X)$ has an abelian Hall $p'$-subgroup, and that $b_X$ is non-nilpotent covered. Then $\mathbb{E}(b) \cong \mathbb{E}(b_X)$, and the hyperfocal subgroups of $b$ and $b_X$ are isomorphic.
\end{lem}

\begin{proof}
Set $H = X \mathrm{C}_G(P)$ and let $(P, e_P)$ be a maximal $b$-Brauer pair. Let $b_H = (e_P)^H$ be the induced block. Then $b_H$ covers $b_X$, and $(P, e_P)$ is a maximal $b_H$-Brauer pair. Since $\mathrm{N}_X(P, e_P) \mathrm{C}_G(P) \leq H$, Lemma~\ref{NGPeP} yields $\mathrm{N}_H(P, e_P) = \mathrm{N}_G(P, e_P)$. By \cite[Proposition 6.1]{A}, we may therefore assume that $G = H$ and $b = b_H$.

Let $L = \mathrm{C}_G(X)$. Since $X$ is normal in $G$ by Corollary \ref{X_i is normal}, $L$ is also normal in $G$. Let $b_L$ be the block of $L$ covered by $b$. Since $b$ is reduced, $b_L$ is $G$-invariant. Set $P_L = P \cap L$. By \cite[Theorem 9.26]{N1}, $P_L$ is a defect group of $b_L$. Note that $X \leq \mathrm{C}_G(P_L)$, so $P_L$ is central in $G$.
Hence \( b_L \) is nilpotent, which forces \( L \leq \mathrm{O}_{p}(G)\mathrm{Z}(G). \)
By Corollary \ref{ZG}, we then obtain \( L \leq \mathrm{Z}(G) \).
Consequently, \( \mathrm{F}^{*}(G) = X \mathrm{Z}(G) \).
Therefore, the quotient group \( G / \mathrm{F}^{*}(G) \) is isomorphic to a subgroup of \( \mathrm{Out}(X) \), and hence is solvable.

Let $N=P \mathrm{F}^*(G)$. By \cite[Lemma 2.4]{CGA}, $N/ \mathrm{F}^*(G)$ is a Sylow $p$-subgroup of $G / \mathrm{F}^*(G)$.
Note that $G = X {\rm C}_G(P) = \mathrm{F}^*(G) {\rm C}_G(P)$. So  $N$ is normal in $G$. Since $\mathrm{Out}(X)$ has an abelian Hall $p'$-subgroup, so does $G / \mathrm{F}^*(G)$. Therefore, the quotient group $G/N$ is an abelian $p'$-group.

Let $b_N$ and $b_{\mathrm{F}^*(G)}$ be the blocks of $N$ and $\mathrm{F}^*(G)$ covered by $b$, respectively. By \cite[Proposition 2.5]{EL2}, $b_N$ and $b_{\mathrm{F}^*(G)}$ have isomorphic inertial quotients and the same hyperfocal subgroup. As $\mathrm{F}^*(G) = X \mathrm{Z}(G)$, the same is true for $b_{\mathrm{F}^*(G)}$ and $b_X$. Therefore, $b_N$ and $b_X$ have isomorphic inertial quotients and the same hyperfocal subgroup.

If $\mathrm{C}_G(P) \leq N$, then $G = N$ and the lemma follows trivially. Thus we may assume  $\mathrm{C}_N(P) \neq \mathrm{C}_G(P)$. We now argue the same as the proof of \cite[Theorem 4.12]{EM}.

Since $G/N$ is an abelian $p'$-group, we construct inductively a chain of normal subgroups
\[
N = K_0 \lneqq K_1 \lneqq K_2 \lneqq \dots \lneqq K_n
\]
as follows. Set $K_0 = N$ and let $b_0 = b_N$. Suppose $K_i$ has been defined. Let $b_i$ be the unique block of $K_i$ covered by $b$, and let $(P, e_{i,P})$ be the maximal $b_i$-Brauer pair such that $e_{i,P}$ covers $e_{i-1,P}$ (for $i \geq 1$). Denote by $\mathrm{C}_G(P, e_{i,P})$ the stabilizer of $e_{i,P}$ in $\mathrm{C}_G(P)$. If $\mathrm{C}_G(P, e_{i,P}) = \mathrm{C}_{K_i}(P)$, the construction stops. Otherwise, set
\[
K_{i+1} = K_i \mathrm{C}_G(P, e_{i,P}).
\]
Then $K_{i+1}$ is a normal subgroup of $G$ strictly containing $K_i$, and the process terminates after finitely many steps.

From the construction we have:
\[
\mathrm{C}_{K_{i+1}}(P, e_{i,P}) = \mathrm{C}_{K_{i+1}}(P), \quad
\mathrm{N}_{K_{i+1}}(P, e_{i,P}) = \mathrm{N}_{K_i}(P, e_{i,P}) \mathrm{C}_{K_{i+1}}(P),
\]
and $e_{i,P}$ is the unique block of $\mathrm{C}_{K_i}(P)$ covered by $e_{i+1,P}$. Consequently,
\[
\mathrm{N}_{K_{i+1}}(P, e_{i+1,P}) \subseteq \mathrm{N}_{K_{i+1}}(P, e_{i,P}).
\]

Let $\mathbb{E}(b_i) = \mathrm{N}_{K_i}(P, e_{i,P}) / \mathrm{C}_{K_i}(P)$ be the inertial quotient of $b_i$. The inclusion above induces a group monomorphism $\iota_i: \mathbb{E}(b_{i+1}) \hookrightarrow \mathbb{E}(b_i)$. The composition $\iota = \iota_1 \circ \iota_2 \circ \cdots \circ \iota_{n-1}$ yields a group monomorphism from $\mathbb{E}(b_n)$ to $\mathbb{E}(b_N)$. Since $b_X$ (and hence $b_N$) is non-nilpotent covered, we have that the order $|\mathbb{E}(b_n)| > 1$. However, $\mathbb{E}(b_N)$ has prime order, which implies that $\mathbb{E}(b_n) \cong \mathbb{E}(b_N)$ and that all $\iota_i$ are group isomorphisms. Therefore, the actions of $\mathbb{E}(b_i)$ on $P$ are identical, and the hyperfocal subgroup of $b_n$ coincides with that of $b_N$.

Recall that $(P, e_P)$ is a maximal $b$-Brauer pair. Since $b$ covers $b_n$, we may assume that $e_P$ covers $e_{n,P}$. If $G = K_n$, the lemma follows. Assume now $K_n\lneqq G$. By the proof of \cite[Lemma 4.11(2)]{EM}, the inclusion $\mathrm{N}_{K_n}(P, e_{n,P}) \subseteq \mathrm{N}_G(P, e_P)$ induces a group monomorphism $\mathbb{E}(b_n) \hookrightarrow \mathbb{E}(b)$. Since both $\mathbb{E}(b_n)$ and $\mathbb{E}(b)$ have prime order, this map is an isomorphism. Hence the actions of $\mathbb{E}(b_n)$ and $\mathbb{E}(b)$ on $P$ are the same.

Summarizing the above, we have $\mathbb{E}(b) \cong \mathbb{E}(b_X)$, and the hyperfocal subgroups of $b$ and $b_X$ are isomorphic.
\end{proof}

\section{Proof of the Main Results}

This section is devoted to the proofs of Theorem~\ref{Main1} and Corollary~\ref{Main2}. Throughout this section, we set $p=2$. We begin by recalling the classification of blocks of quasi-simple groups with abelian defect groups from \cite{EKK}.

\begin{thm}[{\cite[Theorem 6.1]{EKK}}]\label{2-block-abel}
Let $X$ be a quasi-simple group. If $c$ is a block of $X$ with abelian defect group $D$,
then one (or more) of the following holds:
\begin{enumerate}[itemsep=-3pt,topsep=7pt,label=\emph{(\roman*)}]
\item $X/{\rm Z}(X)$ is one of $A_1(2^a)$, ${}^{2}G_2(q)$ (where $q \geq 27$ is a power of $3$ with odd
exponent), or $J_1$, $c$ is the principal block and $D$ is elementary abelian;
\item $X$ is  $Co_3$, $c$ is a non-principal block and $D$ is an elementary abelian group
of rank $3$ (there is one such block);
\item $c$ is Morita equivalent to a block $d$ of a subgroup $L=L_0\times L_1$ of $X$ as following:
The defect group of $d$ is isomorphic to $D$, $L_0$ is abelian and the block of $L_1$ covered by $d$ has Klein four defect groups.
\item $c$ is nilpotent covered.
\end{enumerate}
\end{thm}

Let $b$ be a block of a finite group $G$ with defect group $P$. Assume that $b$ is reduced, that $P$ is abelian, that $G$ has a component $X$, and that the inertial quotient $\mathbb{E}(b)$ has prime order. By Corollary~\ref{X_i is normal}, $X$ is normal in $G$. Denote by $b_X$ the unique block of $X$ covered by $b$, and assume further that $b_X$ is non-nilpotent covered.

\begin{lem}\label{G is Klein}
Keep the notation and assumptions above. If the block $b_X$ has a Klein four hyperfocal subgroup,
then the hyperfocal subgroup of $b$ is also a Klein four group.
\end{lem}

\begin{proof}
Let $D = P \cap X$. Then $D$ is a defect group of the block $b_X$. By Lemma \ref{NGPeP}, the hyperfocal subgroup $Q$ of $b$ with respect to $(P,e_P)$ is contained in $D$. Suppose that $D$ is a Klein four group. Note that $b$ is not nilpotent. So it follows from \cite[Lemma 3.7]{Ta} that the hyperfocal subgroup of $b$ is equal to $D$.

Thus we may assume that $|D| > 4$. Since $b_X$ is non-nilpotent covered and has a Klein four hyperfocal subgroup,
it follows from \cite[Corollary 4.8]{EM} that the inertial quotient $\mathbb{E}(b_X)$ has order $3$, and that
the block $b_X$ occurs only in case (iii) of Theorem~\ref{2-block-abel}.

By the last paragraph of the proof of \cite[Theorem 6.1]{EKK} and \cite[Proposition 5.3]{EKK}, it follows that $X/{\rm Z}(X)$ is a simple group of type $E_7(q)$ or $D_n(q)$ for $q$ an odd prime power and $n/2$ odd. Furthermore, by \cite[Theorem 2.5.12]{GLS}, $\mathrm{Out}(X)$ has an abelian Hall $2'$-subgroup. Therefore, the lemma now follows by applying Lemma~\ref{Iso}.
\end{proof}

\begin{lem}\label{G=XR}
Keep the notation and assumptions above. If $X/{\rm Z}(X)$ is isomorphic to $A_1(2^a)$ and $b_X$ is the principal block,
then $G$ is isomorphic to the direct product of $A_1(2^a)$ and an abelian $2$-group.
Furthermore, $2^a-1$ is  prime.
\end{lem}

\begin{proof}
By \cite[Corollary 4.8]{EM} and the assumption that $X/{\rm Z}(X) \cong A_1(2^a)$, it follows that $X$ is isomorphic to $A_1(2^a)$  and that the inertial quotient $\mathbb{E}(b_X)$ has order $2^a-1$.

Let $L = {\rm C}_G(X)$. Since $X$ is normal in $G$ by Corollary \ref{X_i is normal}, $L$ is also normal in $G$. Since $X$ is simple, the product $H = XL$ is a direct product. Obviously, $H$ is a normal subgroup of $G$, so the quotient group $G/H$ embeds into $\mathrm{Out}(X)$.

Let $D = P \cap X$. Then $D$ is a defect group of $b_X$. Since $b_X$ is the principal block, $D$ is a Sylow $2$-subgroup of $X$. By the Frattini argument, we have $G = X {\rm N}_G(D)$. Furthermore, the subgroup $X {\rm C}_G(D)$ is normal in $G$. Obviously, ${\rm C}_G(P) \leq {\rm C}_G(D)$. Applying Lemma \ref{G=N}, we conclude that $G = X {\rm C}_G(D)$.
Now set $\bar{G} = G/L$, and let $\bar{D}$ and $\bar{X}$ be the images of $D$ and $X$ in $\bar{G}$, respectively. Then $\bar{G} = \bar{X} {\rm C}_{\bar{G}}(\bar{D})$. We identify $\bar{G}$ with a subgroup of $\mathrm{Aut}(X)$ containing $X$ as a normal subgroup.
It follows form \cite[Lemma 4.10]{EM} that $\bar{G} = \bar{X}$. Hence, $G=X\times L$.

Let $b_L$ be the block of $L$ such that $b = b_X \otimes b_L$. Then the corresponding inertial quotient satisfies $\mathbb{E}(b) \cong \mathbb{E}(b_X) \times \mathbb{E}(b_L)$. We know that $|\mathbb{E}(b_X)| = 2^a - 1$ and, by assumption, $|\mathbb{E}(b)|$ is a prime number. Therefore, $2^a - 1$ must be prime and $\mathbb{E}(b_L)$ must be trivial. Consequently, by \cite[Proposition 8.11.3]{L}, the block $b_L$ is nilpotent.

Finally, since $b$ is reduced and $b_L$ is nilpotent, it follows that $L$ is an abelian $2$-group. This completes the proof of the lemma.
\end{proof}

\begin{lem}\label{cannot}
Keep the notation and assumptions above. Then $X/{\rm Z}(X)$ cannot be isomorphic to a simple group of type ${}^{2}G_2(q)$ (where $q \geq 27$ is a power of $3$ with an odd exponent), $J_1$, or $Co_3$.
\end{lem}

\begin{proof}
Suppose, for a contradiction, that $X/{\rm Z}(X)$ is isomorphic to one of the groups listed. By \cite[Corollary 4.8]{EM}, the group $X$ itself is isomorphic to a simple group of type ${}^{2}G_2(q)$ (with $q=3^m$, $m$ odd, $m\ge3$), $J_1$, or $Co_3$.
Furthermore, the inertial quotient $\mathbb{E}(b_X)$ is isomorphic to a Frobenius group of the form $C_7 \rtimes C_3$.
Let $L={\rm C}_G(X)$, which is normal in $G$ since $X$ is. As $X$ is simple, the product $XL$ is  direct.

Suppose that $X$ is isomorphic to $J_1$ or $Co_3$.
By \cite[Tables 5.3j and 5.3f]{GLS}, $\mathrm{Out}(X)$ is trivial. Thus, $G = X \times L$. Let $b_L$ be a block of $L$ such that $b = b_X \otimes b_L$. Then $\mathbb{E}(b) \cong \mathbb{E}(b_X) \times \mathbb{E}(b_L)$. So $\mathbb{E}(b)$ cannot have prime order, contradicting our assumptions.

Suppose that $X$ is isomorphic to ${}^2G_2(q)$.
By \cite[Theorem 2.5.12]{GLS}, $\mathrm{Out}(X)$ is a cyclic group of odd order.
Let $\bar{G} = G/L$, which we identify with a subgroup of $\mathrm{Aut}(X)$ containing $X$ as a normal subgroup. Let $\bar{b}_X$ be the block of $\bar{G}$ covering $b_X$. Since the trivial character of $X$ extends to $\bar{G}$ and all automorphisms of the block algebra $\mathcal{O}Xb_X$ induced by $\bar{G}$ are inner \cite[Theorem 3.1]{ECW2016}, it follows from \cite[Theorem 7]{Ku} that there is an $\mathcal{O}$-algebra isomorphism
\[
\mathcal{O}\bar{G}\bar{b}_X \cong \mathcal{O}Xb_X.
\]
As a consequence, the restriction of characters yields a bijection between the irreducible (ordinary and Brauer) characters of
$\bar{b}_X$ and $b_X$.

Let $\tilde{G} = G/X$. Applying \cite[Corollary 6.16]{I1976} and \cite[Corollary 8.20]{N1}, we describe the irreducible characters of $G$ lying over $b_X$ by the following sets:
\[
\{ \alpha\beta \mid \alpha \in {\rm IBr}(\bar{b}_X),\; \beta \in {\rm IBr}(\tilde{G}) \},
~{\rm and}~
\{ \alpha\beta \mid \alpha \in {\rm Irr}(\bar{b}_X),\; \beta \in {\rm Irr}(\tilde{G}) \}.
\]
Fix an element $\gamma\in {\rm IBr}(\bar{b}_X)$,  and let $\Pi $ consist of all $\beta\in{\rm IBr}(\tilde{G})$ such that $\gamma\beta \in {\rm IBr}(b)$.
Applying \cite[Theorem 3.3 and Lemma 3.5]{N1}, it's easy to  check that $\alpha\beta \in \IBr(b)$ for any $\alpha \in {\rm IBr}(\bar{b}_X)$ and any $\beta \in \Pi$. Thus, $|{\rm IBr}(b)|=|{\rm IBr}(\bar{b}_X)|\cdot |\Pi|=|{\rm IBr}(b_X)|\cdot |\Pi|$.

Recall that the Alperin weight conjecture holds for all $2$-blocks with abelian defect groups (see \cite[Theorem 1.2]{HZ2025}).
Applying this to the principal block $b_X$ of $X$, we conclude that
$
|\mathrm{IBr}(b_X)| = 5.
$
On the other hand, the hyperfocal subgroup $Q$ of $b$ is contained in $P \cap X$ by Lemma~\ref{NGPeP}. Since $P \cap X$ is the defect group of $b_X$, it is elementary abelian of order $8$. As $Q = [P, \mathbb{E}(b)]$ and $\mathbb{E}(b)$ has prime order, the order of $\mathbb{E}(b)$ must be $3$ or $7$. This forces $|{\rm IBr}(b)|$ to be either $3$ or $7$, which contradicts the fact that $|{\rm IBr}(b)|=5\cdot |\Pi|$.

In both cases, we reach a contradiction, completing the proof.
\end{proof}

\begin{order}\label{proof}{Proof of Theorem \ref{Main1}}.
\rm{
Let $G$ be a finite group and $b$ a $2$-block of $G$ with abelian defect group $P$. Assume that the inertial quotient $\mathbb{E}(b)$ of $b$ has prime order. By \cite[Proposition 6.1]{A}, we may assume that $b$ is reduced.

If $P$ is normal in $G$, then $b$ is inertial, and we are in case (i) of Theorem \ref{Main1}.

Now suppose that $P$ is not normal in $G$. Then the layer $\mathrm{E}(G)$ is non-trivial by Theorem \ref{Reduced-block}. By Corollary \ref{X_i is normal}, all components of $G$ are normal. If the blocks of the components of $G$ covered by $b$ are all nilpotent covered, then $b$ is inertial by Proposition \ref{Inertial}, and we are again in case (i) of Theorem \ref{Main1}.

Finally, suppose that $G$ has a component $X$ such that the block $b_X$ covered by $b$ is non-nilpotent covered. Since the defect group of $b_X$ is abelian, the pair $(X, b_X)$ must satisfy one of the cases (i), (ii), or (iii) of Theorem \ref{2-block-abel}. In case (iii), the hyperfocal subgroup of
$b_X$ is a Klein four group. The conclusion of Theorem \ref{Main1} therefore follows from the combined results of Lemmas \ref{G is Klein}, \ref{G=XR}, and \ref{cannot}.
\qed
 }
\end{order}

\begin{order}\label{proof}{Proof of Corollary \ref{Main2}}.
\rm{
The conclusion follows by combining Theorem \ref{Main1} with  \cite[Theorem 4.36]{CR} and \cite[Theorem 1.2]{H2025}.
\qed
 }
\end{order}

\end{document}